\documentclass[11pt,a4paper]{article}

\usepackage{amsmath}
\usepackage{amsfonts}
\usepackage{amsthm}
\usepackage[pdftex]{color,graphicx}
\usepackage{graphicx}
\usepackage{color}
\usepackage[colorlinks,linkcolor=blue,citecolor=red,urlcolor=black]{hyperref}

\newtheorem{Proposition}{Proposition}
\newtheorem{Lemma}{Lemma}
\newtheorem{Remark}{Remark}

\newcommand{\dx}{{\mathrm{d}}x}
\newcommand{\dy}{{\mathrm{d}}y}
\newcommand{\Id}{{\mathbf{1}}}
\newcommand{\dom}{{\mathrm{dom}~}}

\DeclareMathOperator{\ext}{\operatorname{ext}}
\DeclareMathOperator{\inter}{\operatorname{int}}

\begin{document}

	\title{Spectral analysis in broken sheared waveguides}
	\author{Diana C. S. Bello and Alessandra A. Verri}

	\date{\today}
	
	\maketitle 
	
	\begin{abstract}

	\end{abstract}

	Let $\Omega \subset \mathbb R^3$ be a broken sheared waveguide, i.e., 
	it is built by translating a cross-section in a constant direction along
	a broken line in $\mathbb R^3$. We prove that the discrete spectrum of the Dirichlet Laplacian operator in $\Omega$ is non-empty and finite.
	Furthermore, we show a particular geometry for $\Omega$ which implies that the total multiplicity
	of the discrete spectrum is equals 1.
	
	\
	
	
	\
	
	
	

	\section{Introduction}

	Spectral properties of the Dirichlet Laplacian operator in unbounded 
	domains of the $n$ dimensional Euclidean space has been extensively studied in the last years. In this context, the question on the existence of discrete spectrum naturally arises; 
	the results are non-trivial and depend on the geometry of the region
	\cite{avishai, bori2, bori1, briet, bulla, carini, carron, duclosfk, dauge06, dauge,  dauge03, dauge02, duclos, pavelduclos, exnersebaL,  exnertater, solomyak, friedsol, gold, davidkriz, davidlu,  davidruled, davidmainpaper,  nazarov2014,  renger}.
	For example,  consider a  tubular neighborhood of a reference curve  in $\mathbb R^3$. 
	If this domain is an infinite cylinder, then the operator has a purely essential spectrum. However, an appropriated bending effect can create 
	discrete eigenvalues \cite{duclosfk, duclos, gold, davidtvb}.
	
	The operator has interesting spectral properties  even in domains with a simpler geometry.
	Let us consider a conical layer $\Omega$ in $\mathbb R^2$ (resp. $\mathbb R^3$), i.e., an infinite region in $\mathbb R^2$ (resp. $\mathbb R^3$) 
	limited by two identical conical surfaces. Denote by $-\Delta_\Omega^D$ the Dirichlet Laplacian operator in $\Omega$.
	The discrete spectrum of $-\Delta_\Omega^D$ has been discussed in several works
	\cite{avishai, carron, dauge06, dauge, dauge03, exnersebaL,  exnertater, nazarov2014}. In the next paragraphs we present an overview of some of the results  and
	the main purpose of this work.

	Suppose that $\Omega$ is sufficiently smooth in the sense that the reference conical surface is smooth except in its vertex.
	In particular, for the two dimensional case, for each  $\theta \in (0, \pi/2)$, consider
	\begin{equation}\label{aintthet}
		V_\theta := \left\{(x,y) \in \mathbb R^2: x \tan \theta < |y| < \left(x+ \frac{\pi}{\sin \theta}\right) \tan \theta \right\};
	\end{equation}
	$V_\theta$ is also called broken strip. In \cite{exnersebaL} the authors investigated the case  $\theta = \pi/4$ and proved that the operator has 
	a unique discrete eigenvalue which is equal to  $0.93$. 
	In \cite{avishai} it was found that
	$-\Delta_{V_{\theta}}^D$ has at least one discrete eigenvalue and 
	more than one for any sufficiently small angle.
	In \cite{dauge} the authors proved that, for each $\theta \in (0, \pi/2)$, $-\Delta_{V_{\theta}}^D$ 
	has at least one discrete eigenvalue; the number of discrete eigenvalues is always finite; this number tends to infinity as $\theta$ approaches to zero.
	In \cite{nazarov2014} the authors proved the existence of a critical angle $\alpha^*$ so that, for all $\theta \in (\alpha^*, \pi/2)$,
	the total multiplicity of the discrete spectrum is one; an asymptotic lower bound
	for the multiplicity as $\theta$ approaches to zero was also found.

	Three dimensional models were studied in \cite{carron, dauge06, dauge03, pavelduclos, exnertater}.
	As an example, 
	for each  $\theta \in (0, \pi/2)$, consider the conical layer $\Sigma_\theta$ in $\mathbb R^3$ given by rotating the planar region
	\[\{(x,y) \in \mathbb R^2: (x,y) \in ((0, \pi \cot \theta) \times (0, x \tan \theta]) \cup (( \pi \cot \theta, \infty) \times (0, \pi))\},\]
	along the axis $y=x \tan \theta$ in $\mathbb R^3$. The results of 
	\cite{carron, pavelduclos} show that the discrete spectrum of 
	$-\Delta_{\Sigma_\theta}^D$ is non-empty and that  its cardinality can exceed any fixed integer for $\theta$ small enough. Later, in \cite{exnertater}, 
	the authors noted that the proofs of \cite{carron, pavelduclos} can be used 
	to ensure that the operator has an infinite
	sequence of discrete eigenvalues. In \cite{dauge03} the authors also analyzed 
	the infiniteness of the discrete eigenvalues and their expansions for $\theta$ small enough.
	In particular, as the meridian domain of $\Sigma_\theta$ is a strip as in (\ref{aintthet}) , the authors discussed 
	how one can pass from a finite number of bound states to an infinite number adding one dimension.

	In \cite{dauge06} the authors studied the Dirichlet Laplacian with a non-smooth conical layer. In that work, the region can be seen 
	as an octant from which another ``parallel" octant is removed. They proved that the discrete spectrum of the operator is non-empty
	and finite. On the question of finiteness or infiniteness of the discrete spectrum, that paper exhibits a significant difference between non-smooth 
	and smooth conical layers in $\mathbb R^3$.
	

	Inspired by the works in the previous paragraphs, we investigate the spectrum of the Dirichlet Laplacian in a ``broken sheared waveguide''.
	Denote by $\{e_1,e_2,e_3\}$ the canonical basis of  $\mathbb{R}^3$. Let $S$ be  a bounded connected open set of 
	$\mathbb R^2$, denote by $y = (y_1,y_2)$ a point of $S$. 
	Take $\beta \in (0, \infty)$ and consider the spatial curve
	\begin{equation}\label{defrefc}
		r_\beta(x)= (x, 0, \beta |x|), \quad x \in \mathbb{R}.
	\end{equation}
	Define the mapping
	\begin{equation}\label{lmas}
		\begin{array}{lcll}
			\mathcal{L}_\beta: & \mathbb{R} \times S  & \longrightarrow & \mathbb{R}^3\\
			& (x,y_1,y_2)          & \longmapsto     & r_\beta(x)+y_1 e_2 + y_2 e_3,
		\end{array}
	\end{equation}
	and the region
	\begin{equation}\label{defot2}	
		\Omega_\beta := \mathcal{L}_\beta(\mathbb{R} \times S).
	\end{equation}
	Geometrically, $\Omega_\beta$ is obtained by translating $S$ along the curve $r_\beta(x)$ so that, at each point
	of it, $S$ is parallel to the plane generated by $\{e_2, e_3\}$. 
	We remark that $\Omega_\beta$ is symmetric with respect to the plane $y_1y_2$, and has a ``corner'' in this plane. On the other hand,
	outside an compact set, $\Omega_\beta$ is the union of two straight tubes.
	We call $\Omega_\beta$ a broken sheared waveguide.


	Denote by $-\Delta^{D}_{\Omega_\beta}$ the Dirichlet Laplacian in $\Omega_\beta$, i.e., the self-adjoint operator associated with
	the quadratic form 
	\begin{equation} \label{fq311}
		Q^D_{\Omega_\beta} (\psi) 
		= \int_{\Omega_\beta} |\nabla \psi|^2 {\rm d}{\bf x}, \quad \dom Q^D_{\Omega_\beta}
		= H_0^1(\Omega_\beta);
	\end{equation}
	${\bf x} = (x, t, z)$ denotes a point of $\Omega_\beta$, and $\nabla \psi$ denotes the gradient of $\psi$.
	The purpose of this paper is to understand how the geometry of  $\Omega_\beta$  influences the spectrum of the operator.
	
	At first, we characterize the essential spectrum of $-\Delta^{D}_{\Omega_\beta}$.
	Denote $\partial_{y_1} := \partial/\partial y_1$,
	$\partial_{y_2} := \partial/\partial y_2$. Consider 
	the two-dimensional operator
	\begin{equation}\label{fourpp0}
		T(\beta):=-\partial_{y_1}^2  - (1+\beta^2) \partial_{y_2}^2,
	\end{equation}
	$\dom T(\beta)=\{v \in H_0^1(S):  T(\beta) v \in L^2(S)\}$.
	Denote by $E_1(\beta)$ the first eigenvalue of $T(\beta)$.
	Since $T(\beta)$ is an elliptic operator with real coefficients, $E_1(\beta)$ is simple; see, e.g., \cite{evans}.

	\begin{Proposition}\label{propress}
		For each $\beta \in (0, \infty)$, one has $\sigma_{ess}(-\Delta_{\Omega_\beta}^D) = [E_1(\beta),\infty)$. 
	\end{Proposition}
	
	The proof of Proposition \ref{propress} is presented in Section \ref{essspec}.
	The next result ensures the existence of discrete spectrum for  $-\Delta_{\Omega_\beta}^D$.

	\begin{Proposition}\label{exidisspc}
		For each $\beta \in (0, \infty)$, one has $\inf \sigma (-\Delta_{\Omega_\beta}^D) < E_1(\beta)$, i.e., 
		$\sigma_{dis}(-\Delta_{\Omega_\beta}^D) \neq \emptyset$. 
	\end{Proposition}

	According to the models discussed at the beginning of this introduction, the ``broken effect'' in  $\Omega_\beta$ induces the question 
	on the number of discrete eigenvalues. This one is answered by the following result.
	

	\begin{Proposition}\label{numberdisceigen}
		For each $\beta \in (0, \infty)$,  $\sigma_{dis}(-\Delta_{\Omega_\beta}^D)$  is finite. 
	\end{Proposition}
	
	Furthermore,
	
	\begin{Proposition}\label{rectangle}
		If $S$ is the rectangle $(a,b) \times (c, d)$, then,  for each $\beta \in (0,\beta^*)$, where
		\begin{equation}\label{rects}
			\beta^* = 	\left\{
			\begin{array}{cc}
				\sqrt{3} R,  & \mbox{if}\; R \leq 2/\sqrt{3}, \vspace{.3em}\\ 
				(1/2)\sqrt{-R^2+3+\sqrt{49+2R^2+ R^4}}, & \mbox{if}\; R > 2/\sqrt{3},
			\end{array}
			\right.
			\quad \quad R = \frac{d-c}{b-a},
		\end{equation}
		the operator $-\Delta_{\Omega_\beta}^D$ has exactly one discrete eigenvalue.
	\end{Proposition}
	
It is worth to note that $\beta \in (0,\beta^*)$ is only a sufficient condition to ensure the uniqueness of the discrete eigenvalue;
see proof of Proposition \ref{prop:interv2} in Section \ref{disspec}.

As a consequence of Proposition \ref{rectangle}, whenever $\beta$ belongs to a suitable interval,
the thickness of $\Omega_\beta$ does not influence the number of discrete eigenvalues of $-\Delta_{\Omega_\beta}^D$.
In fact, for a fixed  $R > 0$, the family of all rectangles $(a,b) \times (c,d)$ so that $(d-c)/(b-a) = R$ contain 
rectangles with arbitrarily area.


	The proofs of Propositions \ref{exidisspc}, \ref{numberdisceigen} and \ref{rectangle} are presented in Section \ref{disspec}.
	
	\begin{Remark}\label{remidn}{\rm 
			The model considered in this work was inspired by \cite{verri}. In that work the author considered a sheared waveguide constructed by the parallel transporting of
			a two-dimensional cross section along a 
			spatial curve. It was presented some conditions to ensure the existence of discrete eigenvalues for the Dirichlet Laplacian in that region.
			However, the results of that work do not apply under the condition (\ref{defrefc}).}
	\end{Remark}
	
	This work is organized as follows. In Section \ref{essspec} we study the essential spectrum of $-\Delta_{\Omega_\beta}^D$.
	Section \ref{disspec} is dedicated to study the discrete spectrum of the operator. In some situations, when there is no doubt about the domain, 
	the Laplacian operator is simply denoted by $-\Delta$.

	\section{Essential spectrum}\label{essspec}

	For each $\beta \in (0, \infty)$, define the spatial curve
	$\tilde{r}_\beta(x): = (x, 0, \beta x)$, $x \in \mathbb{R}$, and the mapping
	$\tilde{{\cal L}}_\beta(x, y_1, y_2):=  \tilde{r}_\beta(x) + y_1 e_2 + y_2 e_3$, $(x, y_1, y_2) \in \mathbb R \times S$. Consider the straight waveguide
	$\tilde{\Omega}_\beta := \tilde{\mathcal{L}}_\beta(\mathbb{R} \times S)$.
	Proposition 3 of \cite{verri} shows that the Dirichlet Laplacian in $\tilde{\Omega}_\beta$ has a purely essential spectrum and it is equal to the interval
	$[E_1(\beta), \infty)$. Since the essential spectrum of the Dirichlet Laplacian in tubular domains is determined 
	by the geometry of the region at infinity only, the statement of Proposition \ref {propress} in the Introduction is expected. 
	However, we  describe the mains ideas of its proof below.
	
	\vspace{0.3cm}
	
	\noindent
	{\bf Proof of Proposition \ref {propress}:}
	Let $K \subset \Omega_\beta$ be a compact set so that $\Omega_{\ext} := \Omega_\beta \backslash K = \Omega_{\ext}^1 \cup \Omega_{\ext}^2$ (disjoint union), where
	$\Omega_{\ext}^1$ and  $\Omega_{\ext}^2$ are isometrically affine to a straight half-waveguide. Define $\Omega_{\inter} := int(K)$.
	Consider the quadratic form 
	$Q_{\inter}^{DN} \oplus Q_{\ext}^{DN}$, where
	\[Q_{\omega}^{DN} (\psi) = \int_{\Omega_{\omega}}  |\nabla \psi|^2 {\rm d}{\bf x},\]
	\[\dom Q^{DN}_{\omega} = \{\psi \in H^1 (\Omega_{\omega}): \psi = 0 \,\, {\rm in} \,\, \partial \Omega_ \beta \cap \partial \Omega_{\omega} \},\]
	$\omega \in \{\inter, \ext\}$. 
	Denote by  $H_{int}^{DN}$ and $H_{\ext}^{DN}$ the self-adjoint operators associated with $Q_{\inter}^{DN}$ and $Q_{\ext}^{DN}$, respectively. 
	One has  
	\begin{equation}
		-\Delta_{\Omega_\beta}^D  \geq H_{\inter}^{DN} \oplus H_{\ext}^{DN},
	\end{equation}
	in the quadratic form sense. 
	
	Proposition 3 and Remark 1 in \cite{verri} imply that $\sigma_{ess}(H_{\ext}^{DN})= [E_1(\beta), \infty)$. By minimax principle, and since the spectrum of 
	$H_{\inter}^{DN}$ is purely discrete, we have
	\begin{equation*}
		\inf \sigma_{ess} (-\Delta_{\Omega_\beta}^D)  \geq \inf \sigma_{ess}(H_{int}^{DN} \oplus H_{\ext}^{DN}) =  \inf \sigma_{ess}(H_{\ext}^{DN}) = E_1(\beta),
	\end{equation*}
	i.e., $\sigma_{ess} (-\Delta_{\Omega_\beta}^D) \subseteq [E_1(\beta), \infty)$.
	
	Again, by the considerations in \cite{verri}, it is possible to construct a Weyl sequence supported in $\Omega_{\ext}$ associated with $\lambda \geq E_1(\beta)$.
	This shows that $[E_1(\beta), \infty) \subseteq \sigma_{ess} (-\Delta_{\Omega_\beta}^D)$.  \qed

	\section{Discrete spectrum}\label{disspec}
	
	This section is dedicated to investigate the discrete spectrum of $-\Delta_{\Omega_\beta}^D$. 
	At first, we fix some notations that will be useful later.
	Let $Q$ be a closed and lower bounded sesquilinear form with domain $\dom Q$ dense in a Hilbert space $H$.  
	Denote by $A$ the selfadjoint operator associated with $Q$. The Rayleigh quotients of $A$ can de defined as
	\begin{equation}\label{rayquoin}
		\lambda_j(A) = \inf_{\substack{G \subset \dom Q \\ \dim G=j}}   \sup_{\substack{ \psi \in G \\ \psi \neq 0}}  \frac{Q (\psi)}{ \|\psi \|^2_{H}}.
	\end{equation}
	Let $\mu = \inf \sigma_{ess} (A)$. The Rayleigh quotients $\{\lambda_j(A)\}_{j \in \mathbb N}$ form a non-decreasing sequence which satisfies (i) If $\lambda_j(A) < \mu$, then it is a discrete eigenvalue of $A$; (ii) If $\lambda_j(A) \geq \mu$, then 
	$\lambda_j(A) = \mu$; (iii) The $j$-th eigenvalue of $A$ less than $\mu$ (it it exists) coincides with $\lambda_j(A)$.
	We use the notation ${\cal N}(A, \lambda)$ (or ${\cal N}(Q, \lambda)$) to indicate the maximal index $j$ such the $j$-th Rayleigh quotient of 
	$A$ is less than $\lambda$.

	\subsection{Auxiliary problem}\label{auxpropl}
	
	In this subsection we show how the symmetry of $\Omega_\beta$ influences
	the discrete spectrum of $-\Delta_{\Omega_\beta}^D$.
	The arguments are based on \cite{dauge}.
	
	Define 
	\[\Omega_{\beta}^+ : = \Omega_{\beta} \cap \{{\bf x} = (x,t,z) \in \mathbb R^3: x >  0\},\]
	and $\partial_D \Omega_{\beta}^+ := \partial \Omega_{\beta} \cap \partial \Omega_{\beta}^+$.
	Consider the quadratic form
	\[Q_{\Omega_{\beta}^+}^{DN}(\psi) = \int_{\Omega_{\beta}^+} |\nabla \psi|^2 {\rm d}{\bf x},\]
	\[\dom Q_{\Omega_{\beta}^+}^{DN} = \{ \psi \in H^1(\Omega_{\beta}^+): \psi = 0 \,\, {\rm in} \,\, \partial_D \Omega_{\beta}^+\}.\]
	Denote by $-\Delta^{DN}_{\Omega_{\beta}^+}$ the self-adjoint operator associated with $Q_{\Omega_{\beta}^+}^{DN}$. Namely,
	\begin{equation*}
		\dom (-\Delta^{DN}_{\Omega_{\beta}^+}) = \left\{ \psi \in H^1(\Omega_{\beta}^+) : \Delta \psi \in L^2 (\Omega_{\beta}^+), 
		\quad \psi=0 \, \, {\rm in} \,\,\partial_D \Omega_{\beta}^+, \,\, {\rm and} \,\, \partial \psi / \partial x  =0 \,\, {\rm in} \,\,  x=0  \right\}.
	\end{equation*}

	\begin{Lemma}\label{oddevend}
		Suppose $\lambda$ a discrete eigenvalue of  $-\Delta^{D}_{\Omega_{\beta}}$, and denote by $\psi_\lambda$ the corresponding eigenfunction. Then, 
		$\psi_\lambda$ is even, with respect to $x$, and $\partial \psi_\lambda/ \partial x = 0$ on $\Omega_\beta \cap \{{\rm plane} \, x=0\}$.
	\end{Lemma}
	\begin{proof}
		At first, we note that $\lambda < E_1(\beta)$.
		Consider the decomposition $\psi_\lambda = \psi_\lambda^{even} +\psi_\lambda^{odd}$ with respect to $x$. A straightforward calculation shows that
		\begin{equation*}
			-\Delta \psi_\lambda^{even} = \lambda \psi_\lambda^{even}, \quad \mbox{and} \quad 	- \Delta \psi_\lambda^{odd} = \lambda \psi_\lambda^{odd},
		\end{equation*}
		in $\Omega_\beta$. Furthermore, $\psi_\lambda^{odd} = 0$ and $\partial \psi_\lambda^{even} / \partial x = 0$ on $\Omega_\beta \cap \{{\rm plane} \, x=0\}$.	
		We affirm that $\psi_\lambda^{odd} = 0$. In fact, suppose that $\psi_\lambda^{odd} \neq 0$. Then, $\lambda$ is an eigenvalue of $-\Delta_{\Omega_\beta^+}^D$.
		Since $\Omega_{\beta}^+ \subset \tilde{\Omega}_{\beta}$, we have the relation
		\[\lambda_j (-\Delta^{D}_{\Omega_\beta^+}) \geq \lambda_j (-\Delta^{D}_{\tilde{\Omega}_{\beta}}), \quad \forall \; j \geq 1.\]
		However, we know that $\sigma(-\Delta^{D}_{\tilde{\Omega}_{\beta}})=[E_1(\beta), \infty)$. It follows that $\lambda \geq E_1(\beta)$, which is a contradiction.
		Thus, $\psi^{odd} =0$ and $\psi = \psi^{even}$.
	\end{proof}

	\begin{Proposition}\label{edcoin}
		For each $\beta \in (0, \infty)$, one has  $\sigma_{dis}(-\Delta^{D}_{\Omega_{\beta}}) = \sigma_{dis}(-\Delta^{DN}_{\Omega_{\beta}^+})$.
	\end{Proposition}
	
	\begin{proof}
		At first, suppose that $\lambda$ is a discrete eigenvalue of $-\Delta^{DN}_{\Omega_{\beta}^+}$, and denote by $\psi_\lambda^+$ the corresponding eigenfunction.
		Denote by $\psi_\lambda$ the even extension, with respect to $x$, of $\psi_\lambda^+$ to $\Omega_\beta$. Namely,
		\begin{equation*}
			\psi_\lambda (x,t,z) =	\left\{
			\begin{array}{cl}
				\psi_\lambda^+(x,t,z),  & if \; x>0,\\ 
				\psi_\lambda^+(-x,t,z),  & if  \; x<0.
			\end{array}
			\right.
		\end{equation*}
		This function defines a eigenfunction of  $-\Delta^{D}_{\Omega_\beta}$ with eigenvalue  $\lambda$. 
		Then, $\sigma_{dis}(-\Delta^{DN}_{\Omega_{\beta}^+}) \subseteq \sigma_{dis}(-\Delta^{D}_{\Omega_{\beta}})$.
		
		Now, let $\lambda$ be a discrete eigenvalue of $-\Delta^{D}_{\Omega_{\beta}}$, and denote by $\psi_\lambda$ the corresponding eigenfunction.
		By Lemma \ref{oddevend}, the restriction of $\psi_\lambda$ in $\Omega_\beta^+$ defines an eigenfunction of 
		$-\Delta^{DN}_{\Omega_{\beta}^+}$ with eigenvalue $\lambda$.
		Then, $\sigma_{dis}(-\Delta^{D}_{\Omega_{\beta}}) \subseteq \sigma_{dis}(-\Delta^{DN}_{\Omega_{\beta}^+})$.
		
	\end{proof}

	\subsection{Change of coordinates}

	Due to Proposition \ref{edcoin}, we start studying $\sigma_{dis} (-\Delta_{\Omega_\beta^+}^{DN})$ instead of 
	$\sigma_{dis} (-\Delta_{\Omega_\beta}^D)$. More precisely, $-\Delta_{\Omega_\beta^+}^{DN}$ is the self-adjoint operator associated with the quadratic form
	\[Q^{DN}_{\Omega_\beta^+} (\psi) 
	= \int_{\Omega_\beta^+} |\nabla \psi|^2 {\rm d} {\bf x},\]
	\[\dom Q_{\Omega_\beta^+}^{DN} = \{\psi \in H^1(\Omega_\beta^+): \psi = 0 \,\, {\rm in} \,\, \partial_D \Omega_\beta^+\}.\]
	In this subsection we perform a change of coordinates so that the domain  $\dom Q_{\Omega_\beta^+}^{DN}$ does not depend on $\beta$.

	Recall the mapping ${\cal L}_\beta$ given by (\ref{lmas}) in the Introduction. Denote $\Lambda := (0, \infty) \times S$.
	Namely, $\Omega_\beta^+ = \mathcal{L}_\beta(\Lambda)$.
	By Proposition 1 in \cite{verri}, ${\cal L}_\beta^+ : = {\cal L}_\beta |_\Lambda$ is a local $C^{0,1}$-diffeomorphism.
	Since ${\cal L}_\beta^+$ is injective we obtain a
	global $C^{0,1}$-diffeomorphism. Thus,
	the region $\Omega_\beta^+$ can be identified with the Riemannian manifold
	$(\Lambda, G_\beta)$, where $G_\beta=(G_\beta^{ij})$ is the metric induced by ${\cal L}_\beta^+$, i.e.,
	\[G_\beta^{ij} = \langle {\cal G}_\beta^i, {\cal G}_\beta^j \rangle = G_\beta^{ji}, \quad i,j=1,2,\]
	where
	\[{\cal G}_\beta^1 = \frac{\partial {\cal L}_\beta^+}{\partial x}, \quad 
	{\cal G}_\beta^2 = \frac{\partial {\cal L}_\beta^+}{\partial y_1}, \quad {\cal G}_\beta^3 = \frac{\partial {\cal L}_\beta^+}{\partial y_2} .\]
	More precisely,
	\[G_\beta = \nabla {\cal L}_\beta^+ \cdot (\nabla {\cal L}_\beta^+)^t = \left(
	\begin{array}{ccc}
		1 +  \beta^2 &  0  &  \beta \\
		0                               & 1       &  0 \\
		\beta                                &    0                & 1
	\end{array} \right), \quad \det G_\beta = 1.\]
	Now, 
	we consider the unitary operator
	\begin{equation}\label{unituvar}
		\begin{array}{llll}
			{\cal U}_\beta: &   L^2(\Omega_\beta^+)  &  \to &  L^2(\Lambda) \\
			&    \psi   &  \mapsto  &        \psi \circ {\cal L}_\beta^+
		\end{array},
	\end{equation}
	and, we define
	\begin{align}\label{compaquadravar}
		Q_\beta(\psi) & :=  Q_{\Omega_\beta^+}^{DN}({\cal U}_\beta^{-1} \psi)  
		= \int_\Lambda \langle \nabla \psi, G_\beta^{-1} \nabla \psi \rangle \sqrt{{\rm det}\, G_\beta} \, \dx \dy  \nonumber \\ 
		& =  \int_\Lambda \left(\left|\psi' - \beta \frac{\partial \psi}{\partial y_2} \right|^2 +  |\nabla_y \psi|^2 \right) \dx \dy,
	\end{align}
	\[\dom Q_\beta := {\cal U}_\beta(\dom Q_{\Omega_\beta^+}^{DN}) = \{ \psi \in H^1(\Lambda): \psi = 0 \,\, \hbox{on} \,\, (0,\infty) \times \partial S\},\]
	where 
	$\psi' := \partial \psi / \partial x$,  and $\nabla_y \psi := (\partial_{y_1} \psi, \partial_{y_2} \psi)$.
	Denote by $H_\beta$ the self-adjoint operator associated with the quadratic form
	$Q_\beta(\psi)$.

	\subsection{Proof of the results}
	
	Now, we have conditions to prove Propositions \ref{exidisspc} and \ref{numberdisceigen}. 
	In particular, the proof of Proposition \ref{numberdisceigen} is inspired by \cite{dauge}.
	
	\vspace{0.2cm}
	\noindent
	{\bf Proof of Proposition \ref{exidisspc}:}
	Consider the quadratic form $q_\beta(\psi) := Q_\beta(\psi) - E_1(\beta)\|\psi\|^2_{L^2(\Lambda)}$, $\dom q_\beta = \dom Q_\beta$.
	According to (\ref{rayquoin}) and Proposition \ref {propress}, it is enough
	to show that there exists a non null function $\psi \in \dom q_\beta$ so that $q_\beta(\psi) < 0$.
	
	The first step is to find  a sequence $(\psi_n)_{n \in \mathbb N} \subset \dom q_\beta$ so that $q_\beta(\psi_n) \to 0$, as $n \to \infty$.
	For that, let $w \in C^\infty(\mathbb R)$ be a real-valued function so that $w = 1$ for $x \leq 1$,  and  $w = 0$ for $x \geq 2$. Define, 
	for each  $n \in \mathbb N - \{ 0 \}$,
	\[w_n(x) : = w \left(\frac{x}{n}\right) \quad \mbox{and} \quad \psi_n(x,y) : = w_n(x) \chi (y),\]
	where 
	$\chi$ denotes the normalized eigenfunction correspondingly to the eigenvalue $E_1(\beta)$. In particular,
	\begin{equation}\label{estdericuoff}
		\int_0^\infty |w'_n|^2 \dx = \frac{1}{n} \int_0^\infty |w'|^2 \dx \to 0, \quad \hbox{as} \quad n \to \infty,
	\end{equation}
	and 
	\begin{equation}\label{firsteigcross}
		\int_S \left(|\partial_{y_1} \chi|^2   + (1+\beta^2) |\partial_{y_2} \chi|^2 \right) \dy = E_1(\beta).
	\end{equation}

	By (\ref{firsteigcross}), and since $\int_S \chi \partial_{y_2} \chi \dy =0$,
	one has
	\begin{align*}
		q_\beta(\psi_n) 
		& = 
		\int_{\Lambda} \left(|w'_n\chi - \beta w_n \partial_{y_2} \chi|^2 + |w_n|^2 |\nabla_y \chi|^2 - E_1(\beta) |w_n|^2 |\chi|^2 \right) \dx\dy \\
		& = \int_{\Lambda} \left[|w'_n|^2 |\chi|^2 - 2 \beta w_n w'_n \chi \partial_{y_2} \chi 
		+ |w_n|^2 \left(|\partial_{y_1} \chi|^2 +  (1+ \beta^2) |\partial_{y_2}\chi|^2 - E_1(\beta) |\chi|^2 \right) \right]\dx \dy \\
		& =  \int_0^\infty |w'_n|^2  \dx.
	\end{align*}
	By (\ref{estdericuoff}), we can see that  $q_\beta(\psi_n) \to 0$, as $n \to \infty$.

	Now, fix $\varepsilon >0$. For each $n \in \mathbb N - \{0\}$, define
	\[ \psi_{n,\varepsilon}(x,y) := \psi_n(x,y) + \varepsilon \phi(x,y),\]
	for some $\phi \in \dom q_\beta$.
	In this case, 
	\[q_\beta(\psi_{n,\varepsilon}) = q_\beta(\psi_n) + 2 \varepsilon \,  {\rm Re} \, (q_\beta(\psi_n, \phi))
	+ \varepsilon^2 q_\beta(\phi).\]
	The strategy is to show that there exists
	$\phi$ satisfying 
	\begin{equation}\label{rbetaneez}
		\lim_{n \to \infty} q_\beta(\psi_n, \phi) \neq 0.
	\end{equation}	
	In fact, if (\ref{rbetaneez})  holds true, it is enough to choose $\varepsilon$ such that
	$q_\beta(\psi_{n, \varepsilon}) < 0$, for some $n$ large enough.

	Consider $\eta \in C_0^\infty(\mathbb R)$, with ${\rm supp} \, \eta \subset [0,1)$, and $\eta(0) \neq 0$.
	Take
	$h \in H_0^1(S)$ ($h$ will be chosen latter). Define $\phi(x,y) := \eta(x) h(y)$.
	One has
	\begin{align*}
		q_\beta(\psi_n, \phi)
		& =
		\int_{\Lambda} \left[ (w'_n \chi - \beta w_n \partial_{y_2} \chi)(\eta' h -\beta \eta \partial_{y_2} h) + 
		w_n \partial_{y_1} \chi \eta \partial_{y_1} h + w_n \partial_{y_2} \chi \eta  \partial_{y_2} h - E_1(\beta) w_n \chi  \eta h \right]  \dx \dy \\
		& = 
		\int_{\Lambda} (w'_n\chi \eta' h  - \beta w'_n \chi \eta \partial_{y_2} h - \beta w_n \partial_{y_2}  \chi \eta' h +
		\beta^2 w_n \partial_{y_2} \chi  \eta \partial_{y_2} h ) \dx \dy \\
		& + 
		\int_\Lambda \left(w_n \partial_{y_1} \chi \eta  \partial_{y_1} h + w_n \partial_{y_2} \chi \eta  \partial_{y_2} h 
		- E_1(\beta) w_n \chi \eta h \right)  \dx \dy \\
		& = 
		\int_{\Lambda} \left(w'_n\chi \eta' h  - \beta w'_n \chi \eta \partial_{y_2} h - \beta w_n \partial_{y_2}  \chi \eta' h \right) \dx \dy
		-
		\int_\Lambda \left(\partial_{y_1}^2 \chi + (1+\beta^2) \partial_{y_2}^2 \chi - E_1(\beta) \chi \right) w_n \eta h  \dx \dy \\
		& \to - \beta\int_\Lambda  \eta' \partial_{y_2} \chi \, h  \dy \dx.
	\end{align*}
	Finally, take $h(y) := y_2 \chi (y)$. Since $\eta(0) \neq 0$, 
	\[- \beta\int_\Lambda   \eta' \partial_{y_2} \chi \, y_2 \chi  \dy \dx = \beta\eta(0) \int_S \partial_{y_2} \chi \,  y_2 \chi \dy 
	= -\frac{\beta \eta (0)}{2} \neq 0.\]
	Then, (\ref{rbetaneez}) holds true.  \qed
	
	\vspace{0.3cm}

	\noindent
	{\bf Proof of Proposition \ref{numberdisceigen}:}
	Consider a $C^\infty$ partition of unity $(\varphi_0, \varphi_1)$ so that $\varphi_0^2(x) + \varphi_1^2(x) = 1$, 
	$\varphi_0(x) = 1$ for $x<1$, and $\varphi_0(x)=0$ for $x>2$. 
	At first, fix $M>0$. Define 
	$\varphi_{l, M}(x) = \varphi_l\left(x/M\right)$, $l \in \{0,1\}$.
	By IMS localization formula (see, for example, \cite{cycon}), for each $\psi  \in \dom Q_\beta$, 
	\begin{equation}\label{eq:172}
		Q_\beta(\psi) =  Q_\beta(\varphi_{0,M}  \psi) + Q_\beta(\varphi_{1,M}  \psi) - 
		\|\varphi^\prime_{0,M} \psi\|^2_{L^2(\Lambda)} - \|\varphi^\prime_{1,M} \psi\|^2_{L^2(\Lambda)}.
	\end{equation}
	Define $W_{M}(x) := \left| \varphi^\prime_{0}\left(x/M\right) \right| ^2 + \left| \varphi^\prime_{1}\left(x/M\right) \right|^2$. Then,
	\begin{equation}\label{eq:182}
		\|\varphi^\prime_{0,M} \psi\|^2_{L^2(\Lambda)} + \|\varphi^\prime_{1,M} \psi\|^2_{L^2(\Lambda)}
		= \int_{\Lambda} \frac{W_{M}(x)}{M^2} \left( |\varphi_{0,M}\psi|^2 + |\varphi_{1,M}\psi|^2 \right) \dx \dy.
	\end{equation}

	Now, consider the following subsets of $\Lambda$,
	\begin{equation*}
		\mathcal{O}_{0,M} := \{(x,y) \in \Lambda: x <2M\} \quad \hbox{and} \quad \mathcal{O}_{1,M} := \{(x,y) \in \Lambda: x >M\}.
	\end{equation*}
	For $l \in\{0,1\}$, define the quadratic forms 
	\begin{equation}\label{brakquaf}
		Q_{l,M}(\psi) := \int_{\mathcal{O}_{l,M}}  \left( \left| \partial_x \psi - \beta \partial_{y_2} \psi \right|^2 + \left| \partial_{y_1} \psi\right|^2 + \left| \partial_{y_2} \psi\right|^2  -  
		\frac{W_{M}(x)}{M^2}|\psi|^2\right) \dx \dy,
	\end{equation} 
	with
	\begin{gather*}
		\dom (Q_{0,M}) = \{\psi \in H^1(\mathcal{O}_{0,M}) : \psi = 0 \; \mbox{on}\; \partial_D \Lambda \cap \partial \mathcal{O}_{0,M} \; \mbox{and on} \; \{2M\} \times S\},\\
		\dom (Q_{1,M}) = H_0^1(\mathcal{O}_{1,M}).
	\end{gather*}

	By (\ref{eq:172}), (\ref{eq:182}) and (\ref{brakquaf}), we can write, for each $\psi \in \dom Q_\beta$,
	\begin{equation*}
		Q_\beta(\psi) = Q_{0,M} (\varphi_{0,M}\psi) + Q_{1,M} (\varphi_{1,M}\psi).
	\end{equation*}
	
	Now, note that
	\begin{gather*}
		\frac{Q_{0,M}(\varphi_{0,M} \psi) + Q_{1,M}(\varphi_{1,M} \psi)}{\|\varphi_{0,M} \psi\|^2_{L^2(\Lambda)} + \|\varphi_{1,M} \psi\|^2_{L^2(\Lambda)}} 
		\geq 
		\frac{Q_{0,M}(\psi_0) + Q_{1,M}(\psi_1)}{\|\psi_0\|^2_{L^2(\mathcal{O}_{0,M})} + \|\psi_1\|^2_{L^2(\mathcal{O}_{1,M})}}.
	\end{gather*}
	This inequality and the Rayleigh quotient formula imply  
	\[\mathcal{N}(Q_\beta,E_1(\beta)) \leq \mathcal{N}(Q_{0,M},E_1(\beta)) + \mathcal{N}(Q_{1,M},E_1(\beta)).\]
	Since the self-adjoint operator associated with $Q_{0,M}$ has compact resolvent,  $\mathcal{N}(Q_{0,M},E_1(\beta))$ is finite.
	To complete the proof, we are going to show that there exists $M_0 > 0$ so that $\mathcal{N}(Q_{1,M},E_1(\beta))$ is finite, for all $M > M_0$.
	
	Consider the closed subspace ${\cal J}:= \{f(x) \chi (y): f \in L^2(M, \infty)\}$ of the Hilbert space $L^2(\mathcal{O}_{1,M})$, 
	and the orthogonal decomposition
	\begin{equation*}
		L^2(\mathcal{O}_{1,M}) = {\cal J} \oplus  {\cal J}^\perp.
	\end{equation*}
	Then, for $\psi \in L^2(\mathcal{O}_{1,M})$, we write
	\begin{equation} \label{eqdecomp}
		\psi(x,y) = f(x) \chi(y) + \psi_{\perp}(x,y), \quad f \in L^2(M, \infty), \psi_\perp \in {\cal J}^\perp.
	\end{equation}
	In particular, 
	\begin{equation}\label{eqvdodecpo}
		\psi \in \dom Q_\beta \Leftrightarrow f \in H_0^1(M, \infty), \psi_\perp \in \dom Q_\beta \cap {\cal J}^\perp.
	\end{equation}
	Furthermore,  $\psi_{\perp} \in \dom Q_\beta \cap {\cal J}^{\perp}$ implies that 
	\begin{equation} \label{compl}
		\int_{S} \psi_{\perp}(x,y) \chi (y)  \dy = 0, \quad  \hbox{and} \quad
		\int_{S} \partial_x \psi_{\perp}(x,y) \chi (y)  \dy = 0, \quad \mbox{a.e.} \, x.
	\end{equation} 
	
	For $\psi \in \dom Q_\beta$, we use the decomposition in (\ref{eqdecomp}) and the properties in (\ref{eqvdodecpo}) and (\ref{compl}). 
	Then, some straightforward calculations show that 
	\[Q_{1,M} (\psi) =
	Q_{1,M} (f \chi) + Q_{1,M} (\psi_{\perp}) + 4 \beta \, {\rm Re} \,  \int_{\mathcal{O}_{1,M}} \overline{f}' \partial_{y_2}\chi \,\psi_{\perp} \dx \dy.\]

	Fix $\varepsilon > 0$. Since 
	$2 \, {\rm Re}(\overline{f}' \partial_{y_2}\chi \psi_{\perp}) \geq - (1/\varepsilon)|\overline{f}' \partial_{y_2}\chi|^2 - \varepsilon |\psi_{\perp}|^2$, one has
	\begin{equation}\label{eq:10102}
		Q_{1,M} (\psi) 
		\geq  Q_{1,M} (f\chi) + Q_{1,M} (\psi_{\perp}) - \frac{2}{\varepsilon} \int_{M}^{\infty}  \kappa \beta |f'|^2  \dx   
		- 2\varepsilon \beta \int_{\mathcal{O}_{1,M}} |\psi_{\perp}|^2 \dx \dy, 
	\end{equation}
	where $\kappa := \|\partial_{y_2} \chi\|^2_{L^2(S)}$.
	
	Now, we are going to find lower bounds for $Q_{1,M} (f\chi)$ and $Q_{1,M} (\psi_{\perp})$. Again, some calculations show that 
	\begin{align}
		\mathcal{Q}_{1,R} (f \chi) 
		& =  \int_{M}^\infty \left[ |f'|^2 + \left(E_1(\beta) -  \frac{W_{M}(x)}{M^2} \right) |f|^2 \right] \dx \nonumber 
		\\ 
		& \geq \int_{M}^\infty \left[ |f'|^2 + \left(E_1(\beta)  -  
		\frac{\nu}{M^2} \Id_{[M,2M]}\right) |f|^2 \right]\dx,  \label{eq:10302}
	\end{align}
	where $\nu := \|W_M\|_\infty$ does not depend on $M$.

	The inequality
	$- 2 \, {\rm Re} (\partial_x \overline{\psi}_{\perp} \partial_{y_2} \psi_{\perp}) \geq - (1/\varepsilon)|\partial_x \psi_{\perp}|^2 - \varepsilon |\partial_{y_2} \psi_{\perp}|^2$ 
	implies $\left| \partial_x \psi_{\perp} - \beta \partial_{y_2} \psi_{\perp} \right|^2   \geq (\beta^2 - \varepsilon \beta)|\partial_{y_2} \psi_{\perp}|^2$, 
	for $\varepsilon \geq \beta$. 
	Consequently,
	\begin{align*}		
		Q_{1,R} (\psi_{\perp}) 
		& \geq   \int_{\mathcal{O}_{1,M}}  
		\left[ \frac{\beta^2 - \varepsilon \beta +1}{1+\beta^2}\left( (1+\beta^2) |\partial_{y_2} \psi_{\perp}|^2 + | \partial_{y_1} \psi_{\perp}|^2 \right) -  
		\frac{\nu}{M^2}|\psi_{\perp}|^2\right] \dx \dy.
	\end{align*}

	Let $E_2(\beta) > E_1(\beta)$ be the second eigenvalue of $T(\beta)$. By minimax principle, 
	\begin{equation}\label{eq:10202}
		Q_{1,M} (\psi_{\perp})  \geq \left[  \left( \frac{\beta^2 - \varepsilon \beta +1}{1+\beta^2} \right)  E_2(\beta) -  
		\frac{\nu}{M^2}\right] \int_{\mathcal{O}_{1,M}} |\psi_{\perp}|^2 \dx \dy.
	\end{equation}

	We define the  number 
	\[\zeta: = \left(\frac{\beta^2 - \varepsilon \beta +1}{1+\beta^2} \right)  E_2(\beta) 
	- 2 \varepsilon \beta - 1,\]
	and the quadratic forms
	\begin{equation}\label{almle}
		b(f \chi) : = \int_{\sqrt{\nu}}^\infty \left[ \left(1 -  \frac{2 \kappa \beta}{\varepsilon} \right)  |f'|^2 + \left( E_1(\beta)  -  
		\Id_{[\sqrt{\nu},2\sqrt{\nu}]}\right) |f|^2 \right] \dx,
	\end{equation}
	and
	\[s(\psi_\perp) : =  \zeta \int_{\mathcal{O}_{1,\sqrt{\nu}}} |\psi_\perp|^2 \dx \dy,\]
	acting in $\dom Q_\beta \cap {\cal J}$ and $\dom Q_\beta \cap {\cal J}^\perp$, respectively. 

	Suppose $\varepsilon > 2\kappa \beta$ and take $M_0 := \sqrt{\nu}$. Then, for $M > M_0$, combining (\ref{eq:10102}), (\ref{eq:10302}), (\ref{eq:10202}) and (\ref{almle}),  one has
	\[Q_{1,M} (\psi) \geq b(f \chi) + s(\psi_\perp),\]
	for all $\psi \in \dom Q_\beta$.
	This inequality implies that $\mathcal{N}(Q_{1,M},E_1(\beta)) \leq \mathcal{N}(b,E_1(\beta))$. Since $\mathcal{N}(b,E_1(\beta))$ is finite, we complete the proof. \qed

	\vspace{0.3cm}
	\noindent{\bf Number of discrete eigenvalues}
	
	\vspace{0.3cm}
	
	Here, we assume that $S$ is the rectangle $(a, b)  \times (c ,d)$. 
	In this case, 
	\begin{equation}\label{def:2Sgamma}
		\Omega_\beta = \left\lbrace (s,t,z) \in \mathbb{R} \times \left( a,b \right) \times \mathbb{R} : \beta  |s| +c < z < \beta |s| + d \right\rbrace,
	\end{equation}
	and $\Omega_\beta^+$ is isometrically affine to
	\[\hat{\Omega}_\beta  := 
	\left\{(\hat{s}, \hat{t}, \hat{z}) 
	\in \left( -  \frac{\beta (d-c)}{\sqrt{1+\beta^2}}, \infty\right) \times \left( 0, b-a \right) \times \left( 0 ,\frac{(d-c)}{\sqrt{1+\beta^2}}\right):
	\hat{z} < \frac{\hat{s}}{\beta}  +  \frac{(d-c)}{\sqrt{1+\beta^2}}, \; \mbox{if} \; 
	\hat{s} \in \left( -\frac{\beta (d-c)}{\sqrt{1+\beta^2}}, 0 \right) \right\}.\]

	Denote by $-\Delta_{\hat{\Omega}_\beta}^{DN}$ the self-adjoint operator associated with the quadratic form
	\begin{equation} \label{fq4gammas}
		Q^{DN}_{\hat{\Omega}_\beta} (\psi) := \int_{\hat{\Omega}_\beta} |\nabla \psi |^2  {\rm d}\hat{{\bf x}}, \quad
		\dom Q^{DN}_{\hat{\Omega}_\beta} := \{\psi \in H^1 (\hat{\Omega}_\beta): \psi = 0 \,\, {\rm in} \,\, \partial_D \hat{\Omega}_\beta\};
	\end{equation}
	$\hat{{\bf x}} = (\hat{s}, \hat{t}, \hat{z})$ denotes a point of $\hat{\Omega}_\beta$ 
	and $\partial_D \hat{\Omega}_\beta := \mathbb{R} \times \left( 0, b-a \right) \times ( 0,(d-c)/ \sqrt{1+\beta^2}) \cap \partial \hat{\Omega}_\beta$.

	Consider the dilation map $\mathcal{F}_\beta: \mathbb R^3 \to \mathbb R^3$, $\mathcal{F}_\beta (x, y_1, y_2) := ( \sqrt{2} \beta x / \sqrt{1+\beta^2},y_1, \sqrt{2} y_2 /\sqrt{1+\beta^2})$, 
	and define 
	\[\hat{\Omega} := \mathcal{F}_\beta^{-1}(\hat{\Omega}_\beta) =  \left\{ (x, y_1, y_2) \in \left( -  \frac{d-c}{\sqrt{2}}, \infty\right) \times \left( 0   , b-a   \right) \times \left( 0 ,    \frac{d-c}{\sqrt{2}}\right) : \\
	y_2 < x  +   \frac{d-c}{\sqrt{2}}, \; \mbox{if} \; x \in \left( -   \frac{d-c}{\sqrt{2}}, 0 \right) \right\}.\]

	Finally, consider the unitary operator $\hat{{\cal U}}_\beta: L^2 (\hat{\Omega}_\beta)  \to L^2 (\hat{\Omega})$,
	$\hat{{\cal U}}_\beta \psi : =\sqrt{2 \beta/(1+\beta^2)} \, (\psi \circ {\cal F}_\beta)$,
	and define the quadratic form
	\begin{align*}
		\hat{Q}_\beta(\psi) :=  Q_{\hat{\Omega}_\beta}^{DN}(\hat{{\cal U}}_\beta^{-1} \psi)  
		=  \int_{\hat{\Omega}} \left[\left(\frac{1+\beta^2}{2\beta^2}\right) \left|\psi' \right|^2 + |\partial_{y_1} \psi|^2 + 
		\left(\frac{1+\beta^2}{2}\right)|\partial_{y_2} \psi|^2 \right] \dx \dy,
	\end{align*}
	\begin{equation*}
		\dom \hat{Q}_\beta := \hat{{\cal U}}_\beta(\dom Q_{\hat{\Omega}_\beta}^{DN}) = 
		\{ \psi \in H^1({\hat{\Omega}}): \psi = 0 \,\, \hbox{in} \,\, \partial_D \hat{\Omega}\},	
	\end{equation*}
	where $\partial_D \hat{\Omega} := \mathbb{R}  \times ( 0 ,b-a) \times (0, (d-c)/\sqrt{2}) \cap \partial \hat{\Omega}$. 
	Denote by $\hat{H}_\beta$ the self-adjoint operator associated with $\hat{Q}_\beta(\psi)$. 
	
	In particular, by the results of Section \ref{essspec}  and Subsection \ref{auxpropl}, one has 
	\begin{equation}\label{coincdisess}
		\sigma_{ess}(-\Delta_{\Omega_\beta}^D) = \sigma_{ess}(\hat{H}_\beta) = \left[ \pi^2 \left( \frac{1}{(b-a)^2} + \frac{1+\beta^2}{(d-c)^2} \right), \infty \right), \quad \
		\hbox{and} \quad \sigma_{dis}(-\Delta_{\Omega_\beta}^D) = \sigma_{dis}(\hat{H}_\beta).
	\end{equation}

	
	The next result is inspired by \cite{nazarov2014}.

	\begin{Proposition} \label{prop:interv2}
		For each $\beta \in (0,\beta^*)$, where
		\begin{equation*}
			\beta^* = 	\left\{
			\begin{array}{cc}
				\sqrt{3} R,  & \mbox{if}\; R \leq 2/\sqrt{3},\vspace{.3em}\\
				(1/2)\sqrt{-R^2+3+\sqrt{49+2R^2+ R^4}}, & \mbox{if}\; R > 2/\sqrt{3},
			\end{array}
			\right.
			\quad \quad R = \frac{d-c}{b-a},
		\end{equation*}
		the operator $\hat{H}_\beta$ has exactly  one discrete eigenvalue.
	\end{Proposition}

	\begin{proof}
		Define $W$ as the region limited by the faces
		\[T_1 :=  \left\{ (x, y_1, y_2) \in \left( -  \frac{d-c}{\sqrt{2}}, 0 \right) \times \left\lbrace  b-a  \right\rbrace   \times \left( 0 ,   \frac{d-c}{\sqrt{2}} \right) : y_2 < x  +  \frac{d-c}{\sqrt{2}} \right\},\]
		\[T_2 :=  \left\{ (x, y_1, y_2) \in \left( -  \frac{d-c}{\sqrt{2}}, 0 \right) \times \left\lbrace  0  \right\rbrace   \times \left( 0 ,   \frac{d-c}{\sqrt{2}} \right) :y_2 < x  +  \frac{d-c}{\sqrt{2}} \right\},\]
		\[T_3 :=  \left\{ (x, y_1, y_2) \in \left( -  \frac{d-c}{\sqrt{2}}, 0 \right) \times \left(  0 , b-a  \right)   \times \{0\} \right\},\]
		\[T_4 := \left\{ (x, y_1, y_2) \in \{0\} \times \left(  0  , b-a  \right)   \times \left( 0,  \frac{d-c}{\sqrt{2}}\right) \right\}.\]
		Now, consider the operator 
		\[J(\beta):=- \left(\frac{1+\beta^2}{2\beta^2}\right) \partial_{x}^2 -\partial_{y_1}^2  -  \left(\frac{1+\beta^2}{2}\right) \partial_{y_2}^2,\]
		and the auxiliary problem
		\begin{equation}\label{probl:ax4}
			\left\{
			\begin{array}{l}
				J(\beta) \psi = \mu \psi, \,\, \mbox{in} \,\, W,\\ 
				\psi =0, \,\, \mbox{in}\,\, T_1 \cup T_2 \cup T_3,\\	
				\partial_x \psi =0, \,\, \mbox{in} \,\,  T_4,\\
				\partial_{\nu} \psi =0, \,\, \mbox{in} \,\, \partial W \setminus (\overline{T_1} \cup \overline{T_2} \cup \overline{T_3} \cup \overline{T_4}),
			\end{array}
			\right.
		\end{equation}
		where $\partial \nu$ denotes the directional derivative along the exterior normal. 
		
		Denote by $\mu_1^\beta$ and $\mu_2^\beta$ the first two eigenvalues of the problem (\ref{probl:ax4}), and by 
		$\psi_1^\beta$ and $\psi_2^\beta$ the correspondingly eigenfunctions. In particular, $J(1)$ is the Laplacian operator. Suppose that  $R\leq 2/\sqrt{3}$. Then,
		\[\mu_1^1 =   \pi^2 \left( \frac{1}{(d-c)^2} + \frac{1}{(b-a)^2}\right)  , \quad
		\psi_1^1 (x,y_1,y_2)= \frac{2}{A\sqrt{B}} \cos \left( \frac{\pi}{2A} x\right)  \sin \left( \frac{\pi}{2B} y_1 \right)  \sin \left( \frac{\pi}{2A} y_2 \right),\]
		\[\mu_2^1 =  \pi^2 \left( \frac{1}{(d-c)^2} + \frac{4}{(b-a)^2}\right), \quad
		\psi_2^1 (x,y_1,y_2)=  \frac{2}{A\sqrt{B}} \cos \left( \frac{\pi}{2A} x\right)  \sin \left( \frac{\pi}{B} y_1\right)  \sin \left( \frac{\pi}{2A} y_2\right),\]
		where $A=  (d-c)/\sqrt{2}$ e $B=(b-a)/2$. Suppose
		\begin{equation}\label{eq:022}
			\mu_2^\beta \geq \pi^2  \left( \frac{1}{(b-a)^2} + \frac{1+\beta^2}{(d-c)^2} \right).
		\end{equation}
		
		Consider the closed subspace
		\begin{equation*}
			E_\beta^\perp := \left\{ \varphi \in \dom \hat{Q}_\beta: \int_{W} \psi_1^\beta \varphi \, \dx \dy =0 \right\}
			\subset \dom \hat{Q}_\beta.
		\end{equation*}
		
		Inequality (\ref{eq:022}) and minimax principle imply that
		\begin{align}
			\pi^2  \left( \frac{1}{(b-a)^2} + \frac{1+\beta^2}{(d-c)^2} \right) \int_{W}|\varphi|^2 \dx\dy 
			& \leq \mu_2^\beta \int_{W}|\varphi|^2 \dx\dy \nonumber \\
			& \leq  \int_{W} \left[ \left( \frac{1+\beta^2}{2\beta^2}\right)  \left|\partial_{x} \varphi \right|^2 +  |\partial_{y_1} \varphi|^2 + \left( \frac{1+\beta^2}{2}\right) |\partial_{y_2} \varphi|^2 \right] \dx  \dy, \label{des:4}
		\end{align}		
		for all $\varphi \in E_\beta^\perp$. On the other hand, 
		\begin{equation}\label{des:5}
			\pi^2  \left( \frac{1}{(b-a)^2} + \frac{1+\beta^2}{(d-c)^2} \right) \int_{\hat{\Omega} \setminus W}|\varphi|^2 \dx\dy \leq  \int_{\hat{\Omega} \setminus W} \left[ \left( \frac{1+\beta^2}{2\beta^2}\right)  \left|\partial_{x} \varphi \right|^2 +  |\partial_{y_1} \varphi|^2 + \left( \frac{1+\beta^2}{2}\right) |\partial_{y_2} \varphi|^2 \right] \dx  \dy,
		\end{equation}
		for all $\varphi \in \dom \hat{Q}_\beta$.
		Consequently, by (\ref{des:4}) and (\ref{des:5}), one has
		\begin{equation*}
			\pi^2  \left( \frac{1}{(b-a)^2} + \frac{1+\beta^2}{(d-c)^2} \right) \int_{\hat{\Omega}}|\varphi|^2 \dx\dy \leq  \int_{\hat{\Omega}} \left[ \left( \frac{1+\beta^2}{2\beta^2}\right)  \left|\partial_{x} \varphi \right|^2 +  |\partial_{y_1} \varphi|^2 + \left( \frac{1+\beta^2}{2}\right) |\partial_{y_2} \varphi|^2 \right]  \dx  \dy,  
		\end{equation*}
		for all $\varphi \in E_\beta^\perp$.
		Since $E_\beta^\perp$ has codimension $1$, one has
		\begin{equation*}
			\lambda_2(\hat{H}_\beta) \geq  \pi^2  \left( \frac{1}{(b-a)^2} + \frac{1+\beta^2}{(d-c)^2} \right).
		\end{equation*}
		Then, $\lambda_1(\hat{H}_\beta)$ is the unique eigenvalue of $\hat{H}_\beta$ in the interval $(0, \pi^2 (1/((b-a)^2) + (1+\beta^2)/((d-c)^2)))$.

		To complete the proof we need to find  $\beta$ so that (\ref{eq:022}) holds true. Note that
		\begin{align*}
			\mu_2^1 & =  
			\inf_{\substack{F \subset {\cal C} \\ \dim F = 2}} \sup_{\substack{\psi \in F \\ \psi \neq 0}} 
			\frac{\int_{W} \left(\left| \partial_x \psi\right| ^2 +  \left|  \partial_{y_1} \psi \right|^2 +  \left|  \partial_{y_2} \psi \right|^2\right)  \dx  \dy}{\|\psi\|^2_{L^2(W)}}  \\
			& \leq \max \left\{ \frac{2\beta^2}{1+\beta^2}, 1, \frac{2}{1+\beta^2}\right\}  \inf_{\substack{F \subset {\cal C} \\ \dim F = 2}}  \sup_{\substack{\psi \in F \\ \psi \neq 0}} 
			\frac{\int_{W} \left[ \left(\frac{1+\beta^2}{2 \beta^2}\right) \left| \partial_x \psi\right| ^2 +  \left|  \partial_{y_1} \psi \right|^2 + 
				\left(\frac{1+\beta^2}{2}\right) \left|  \partial_{y_2} \psi \right|^2\right] \dx  \dy}{\|\psi\|^2_{L^2(W)}} \\
			& = \max \left\{ \frac{2\beta^2}{1+\beta^2}, 1, \frac{2}{1+\beta^2}\right\} \mu_2^\beta.
		\end{align*}
		Consequently,
		\begin{equation}\label{filyin}
			\mu_2^\beta \geq \pi^2 \left( \frac{1}{(d-c)^2} + \frac{4}{(b-a)^2}\right) \min \left \{ \frac{1+\beta^2}{2 \beta^2}, 1, \frac{1+\beta^2}{2} \right\}.
		\end{equation}
		Finally, for $\beta \in (0, \sqrt{3}R)$, the inequality (\ref{eq:022})  holds true.
		
		Now, assume that  $R > 2/\sqrt{3}$. In this case, 
		\[\mu_1^1 =   \pi^2 \left( \frac{1}{(d-c)^2} + \frac{1}{(b-a)^2}\right), \quad 
		\mu_2^1 = \pi^2 \left( \frac{5}{(d-c)^2} + \frac{1}{(b-a)2}\right).\]
		A similar 
		analysis shows that the result is obtained for $\beta \in \left(0, (1/2)\sqrt{-R^2+3+\sqrt{49+2R^2+R^4}}\right)$.
	\end{proof}

	\vspace{0.2cm}
	\noindent
	{\bf Proof of Proposition \ref{rectangle}:} 
	It follows by (\ref{coincdisess}) and Proposition \ref{prop:interv2}. \qed


	\vspace{0.2cm}
	\noindent
	{\bf Acknowledgments}

	\vspace{0.2cm}
	\noindent
	Diana C. S. Bello was supported by CAPES (Brazil) through the process: 88887.511866/2020-00.

\end{document}